\documentclass[a4paper,11pt]{amsart}

\usepackage{amsmath, amsthm, amsfonts, amssymb}
\usepackage[bookmarks=false]{hyperref}
\usepackage{enumerate}

\usepackage{color}

\numberwithin{equation}{section}

\newtheorem{letterthm}{Theorem}

\newtheorem{lettercor}[letterthm]{Corollary}

\newtheorem{thm}{Theorem}[section]
\newtheorem{lem}[thm]{Lemma}

\newtheorem{prop}[thm]{Proposition}

\theoremstyle{definition}
\newtheorem{rem}[thm]{Remark}

\newtheorem{exam}{Example}

\newtheorem{df}[thm]{Definition}

\newcommand{\R}{\mathbf{R}}
\newcommand{\C}{\mathbf{C}}
\newcommand{\Z}{\mathbf{Z}}

\newcommand{\N}{\mathbf{N}}

\newcommand{\Ad}{\operatorname{Ad}}
\newcommand{\id}{\text{\rm id}}

\newcommand{\Inn}{\operatorname{Inn}}
\newcommand{\Aut}{\operatorname{Aut}}
\newcommand{\Out}{\operatorname{Out}}
\newcommand{\rL}{\mathord{\text{\rm L}}}
\newcommand{\rC}{\mathord{\text{\rm C}}}

\newcommand{\rE}{\mathord{\text{\rm E}}}

\newcommand{\Ball}{\mathord{\text{\rm Ball}}}

\newcommand{\supp}{\mathord{\text{\rm supp}}}

\newcommand{\II}{{\rm II}}
\newcommand{\III}{{\rm III}}

\begin{document}

\title[Fullness of crossed products of factors by discrete groups]{Fullness of crossed products of factors\\ by discrete groups}

\begin{abstract}
Let $M$ be an arbitrary factor and $\sigma : \Gamma \curvearrowright M$ an action of a discrete group. In this paper, we study the fullness of the crossed product $M \rtimes_\sigma \Gamma$. When $\Gamma$ is amenable, we obtain a complete characterization: the crossed product factor $M \rtimes_\sigma \Gamma$ is full if and only if $M$ is full and the quotient map $\overline{\sigma} : \Gamma \rightarrow \Out(M)$ has finite kernel and discrete image. This answers a question of Jones from \cite{Jo81}. When $M$ is full and $\Gamma$ is arbitrary, we give a sufficient condition for $M \rtimes_\sigma \Gamma$ to be full which generalizes both Jones' criterion and Choda's criterion. In particular, we show that if $M$ is any full factor (possibly of type $\III$) and $\Gamma$ is a non-inner amenable group, then the crossed product $M \rtimes_\sigma \Gamma$ is full. 
\end{abstract}

\address{Laboratoire de Math\'ematiques d'Orsay\\ Universit\'e Paris-Sud\\ CNRS\\ Universit\'e Paris-Saclay\\ 91405 Orsay\\ FRANCE}

\author{Amine Marrakchi}
\email{amine.marrakchi@math.u-psud.fr}

\thanks{AM is supported by ERC Starting Grant GAN 637601}



\maketitle


\section{Introduction and statement of the main results}

\subsection*{Introduction}
Let $M$ be any factor and let $\sigma : \Gamma \curvearrowright M$ be an action of a discrete group $\Gamma$. This paper is motivated by the following question: when is the crossed product factor $M \rtimes_\sigma \Gamma$ a \emph{full} factor? Recall that a type ${\rm II_1}$ factor $M$ is \emph{full}, or equivalently does not have property Gamma of Murray and von Neumann \cite{MvN43}, if every uniformly bounded net $(x_i)_{i \in I}$ that is \emph{central}, meaning that $\lim_i \|x_i a - a x_i\|_2 = 0$ for every $a \in M$, must be {\em trivial}, meaning that $\lim_i \|x_i - \tau(x_i)1 \|_2 =0$. More generally, following \cite{Co74}, we say that an arbitrary factor $M$ is {\em full} if every uniformly bounded net $(x_i)_{i \in I}$ in $M$ that is {\em centralizing}, meaning that $\lim_i \|x_i \varphi - \varphi x_i\| = 0$ for all $\varphi \in M_\ast$, must be {\em trivial}, meaning that there exists a bounded net $(\lambda_i)_{i \in I}$ in $\C$ such that $x_i - \lambda_i 1 \to 0$ strongly as $i \to \infty$. 

By \cite[Theorem 3.1]{Co74}, for any full factor $M$, the subgroup of inner automorphisms $\Inn(N)$ is closed in the group of automorphisms $\Aut(M)$ and hence the quotient group $\Out(M) = \Aut(M) / \Inn(M)$ inherits a structure of a complete topological group. In \cite{Jo81}, Jones proved that if $M$ is a full $\II_1$ factor and $\sigma : \Gamma \curvearrowright M$ is an action such that the quotient map $\overline{\sigma} : \Gamma \rightarrow \Out(M)$ is injective and has discrete image, then the crossed product $M \rtimes_\sigma \Gamma$ is also a full factor. This result was recently generalized to arbitrary factors in \cite{Ma16}. Moreover, Jones proved that this sufficient condition is in fact necessary when $\Gamma = \Z$. His proof only works for abelian groups but Jones suggested that this statement should be true for every amenable group. Our first main theorem answers this question affirmatively and gives a complete characterization of the fullness of $M \rtimes_\sigma \Gamma$ when $\Gamma$ is amenable.

\begin{letterthm} \label{main amenable}
Let $M$ be a factor and $\sigma: \Gamma \curvearrowright M$ an action of a discrete group $\Gamma$ such that $M \rtimes_\sigma \Gamma$ is a factor. Suppose that $\Gamma$ is amenable. Then the following are equivalent:
\begin{itemize}
\item [$(\rm i)$] $M \rtimes_\sigma \Gamma$ is full.
\item [$(\rm ii)$] $M$ is full and the map $\overline{\sigma} : \Gamma \rightarrow \Out(M)$ has finite kernel and discrete image.
\end{itemize}
\end{letterthm}

\begin{exam}
Let $G$ be a discrete amenable group and let $\pi : G \rightarrow \mathcal{O}(H_\R)$ be any orthogonal representation. Then the crossed product associated to the \emph{free Bogoljubov action} $\Gamma(H_\R)'' \rtimes_\pi G$ is a full factor if and only if $\dim H_\R \neq 1$, $\pi$ is faithful and $\pi(G)$ is discrete in $\mathcal{O}(H_\R)$. Indeed, by \cite[Theorem A]{H14}, this condition is sufficient and by Theorem \ref{main amenable}, it is necessary. This was previsously only known when $G$ is abelian \cite[Corollary 6.2]{H14}.
\end{exam}

The implication $(\rm ii) \Rightarrow (\rm i)$ of Theorem \ref{main amenable} is always true even without the amenability assumption, so our new input is $(\rm i) \Rightarrow (\rm ii)$. This implication actually follows from a much more general property: no action of a discrete amenable group on a diffuse von Neumann algebra can be strongly ergodic. This property was proved very recently by Popa, Shlyakhtenko and Vaes \cite{PSV18} in the case where the action is \emph{free} and trace preserving. In our Theorem \ref{not strongly ergodic}, we remove the freeness (and the traciality) assumption and the proof of $(\rm i) \Rightarrow (\rm ii)$ follows quite easily.

Note that the implication $(\rm i) \Rightarrow (\rm ii)$ fails when $\Gamma$ is not amenable. First, $M$ need not be full. For example, a Bernoulli crossed product $(A,\varphi)^{\otimes \Gamma} \rtimes \Gamma$ is always full when $\Gamma$ is non-amenable \cite[Lemma 2.7]{VV14}. Moreover, even if we assume that $M$ is full, the map $\overline{\sigma} : \Gamma \rightarrow \Out(M)$ need not have discrete image. In fact, a well-known theorem of Choda \cite{Ch81} shows that if $M$ is a full $\II_1$ factor and $\Gamma$ is a non-inner amenable group (e.g.\ $\Gamma=\mathbb{F}_2$), then for \emph{any} action $\sigma : \Gamma \curvearrowright M$, the crossed product $M \rtimes_\sigma \Gamma$ is a full factor. Recall that $\Gamma$ is inner amenable if it admits a non-trivial conjugacy invariant mean, or equivalently if the action by conjugacy $\Gamma \curvearrowright \beta \Gamma \setminus \{ 1 \}$ admits an invariant probability measure.

In our second main theorem, we give a sufficient condition for the fullness of $M \rtimes_\sigma \Gamma$ which generalizes both criterions of Jones and Choda. In order to formulate it, we associate, to any action $\sigma : \Gamma \curvearrowright M$, the following compact $\Gamma$-space (where $\Gamma$ acts by conjugacy):
$$\partial_\sigma \Gamma=\{ x \in \beta \Gamma \setminus \Gamma \mid \lim_{g \to x} \overline{\sigma}(g)=1 \}.$$
\begin{letterthm} \label{main fullness}
Let $M$ be a full factor and let $\sigma: \Gamma \curvearrowright M$ be an action of a discrete group $\Gamma$ such that $M \rtimes_\sigma \Gamma$ is a factor. Suppose that the action $\Gamma \curvearrowright \partial_\sigma \Gamma$ has no invariant probability measure. Then $M \rtimes_\sigma \Gamma$ is full. 
\end{letterthm}
Note that Theorem \ref{main fullness} applies in particular when $\partial_\sigma \Gamma =\emptyset$ which means precisely that the map $\overline{\sigma} : \Gamma \rightarrow \Out(M)$ has finite kernel and discrete image. So this case was essentially already known by \cite{Jo81} and \cite{Ma16}.

Another extreme case where Theorem \ref{main fullness} applies is when $\Gamma$ is non-inner amenable since $\beta \Gamma$ does not carry any non-trivial invariant probability measure in that case. Therefore, we obtain the following generalization of Choda's theorem \cite{Ch81} to factors of arbitrary type.

\begin{lettercor} \label{main cor}
Let $M$ be any full factor and let $\sigma : \Gamma \curvearrowright M$ be an action of a discrete non-inner amenable group. Then $M \rtimes_\sigma \Gamma$ is a full factor.
\end{lettercor}

We would like to emphasize the fact that Choda's argument of \cite{Ch81} fails in general for non-state preserving actions so that different methods are required to obtain similar results in that case, see for instance \cite{HI15}, \cite{Oz16} or \cite[Corollary F]{HI18}. Note that non-state preserving actions $\sigma : \Gamma \curvearrowright M$ on full type $\III$ factors are abundant. For example, this happens if $\Gamma$ contains a subgroup which acts ergodically on $M$ and preserves a state $\varphi$ which is not preserved by $\Gamma$. Therefore Corollary \ref{main cor} cannot be obtained by just adapting Choda's proof and we use in a crucial way the spectral gap characterization of fullness obtained in \cite{Ma16}.

\tableofcontents

\section{Preliminaries} 
\subsection*{Basic notations}
Let $M$ be any von Neumann algebra. We denote by $M_\ast$ its predual, by $\mathcal U(M)$ its group of unitaries and by $\mathcal Z(M)$ its center. The uniform norm on $M$ is denoted by $\| \cdot \|_\infty$ and the unit ball of $M$ with respect to the uniform norm $\|\cdot\|_\infty$ is denoted by $\Ball(M)$. If $\varphi \in M_*^+$ is a positive functional, we put $\| x \|_\varphi=\varphi(x^*x)^{1/2}$ and $\|x\|_\varphi^\sharp=\left( \varphi(x^*x)+\varphi(xx^*) \right)^{1/2}$ for all $x \in M$. 

\subsection*{Standard form}  Let $M$ be any von Neumann algebra. We denote by $(M, \rL^2(M), J, \rL^2(M)_+)$ the standard form of $M$ \cite{Ha73}. Recall that $\rL^2(M)$ is naturally endowed with the structure of a $M$-$M$-bimodule: we will simply write $x \xi y = x Jy^*J \xi$ for all $x, y \in M$ and all $\xi \in \rL^2(M)$. The vector $J \xi$ will be also simply denoted by $\xi^{*}$ so that $(x\xi)^{*}=\xi^{*}x^{*}$. For every $\xi \in \rL^2(M)$ we denote by $|\xi| \in \rL^2(M)_+$ its positive part. If $\xi \in \rL^2(M)_+$ we denote by $\xi^2 \in M_*^+$ the positive functional on $M$ defined by $x \mapsto \langle x \xi, \xi \rangle$.

For every positive functional $\varphi \in M_\ast$, we let
$$\Sigma_\varphi=\left\{\eta \in \rL^2(M) : \exists \lambda \in \R_+, \;  |\eta|^2 \leq \lambda\varphi  \text{ and }  |\eta^{*}|^2 \leq \lambda\varphi  \right\}=M \varphi^{1/2} \cap \varphi^{1/2}M. $$
Note that if for some (not necessarily bounded) net $(x_i)_{i \in I}$ in $M$ we have $\|x_i \|_\varphi \to 0$, then $\| x_i\xi \| \to 0$ for all $\xi \in \Sigma_\varphi$. See the discussion after \cite[Lemma 3.2]{Ma16} for further details.

\subsection*{Topological groups associated to a von Neumann algebra}
A \emph{topological group} is a group $G$ equipped with a topology making the map 
$ (g,h) \in G \times G \mapsto gh^{-1}$ continuous. If $H$ is a normal subgroup of $G$, then  $G/H$ is a topological group with respect to the quotient topology (not necessarily Hausdorff). A topological group $G$ is said to be \emph{complete} if it is Hausdorff and complete with respect to the uniform structure generated by the following sets
$$ U_\mathcal{V}=\{ (g,h) \in G \times G \mid gh^{-1} \in \mathcal{V} \text{ and } g^{-1}h \in \mathcal{V} \}, $$
where $\mathcal{V}$ runs over the neighborhoods of $1$ in $G$. If $G$ is complete and $H$ is a subgroup of $G$, then $H$ is complete if and only if it is closed in $G$. If moreover $H$ is normal, then $G/H$ is also complete. We say that a complete topological group is Polish if it is separable and completely metrizable.

Let $M$ be a von Neumann algebra. Then the restriction of the weak$^*$ topology, the strong topology and the $*$-strong topology all co\"incide on $\mathcal{U}(M)$ and they turn $\mathcal{U}(M)$ into a complete topological group. If moreover $M_*$ has separable predual, then $\mathcal{U}(M)$ is Polish.

The group $\Aut(M)$ of all $\ast$-automorphisms of $M$ acts on $M_*$ by $\theta(\varphi)=\varphi \circ \theta^{-1}$ for all $\theta \in \Aut(M)$ and all $\varphi \in M_*$. Following \cite{Co74, Ha73}, the $u$-topology on $\Aut(M)$ is the topology of pointwise norm convergence on $M_*$, meaning that a net $(\theta_i)_{i \in I}$ in $\Aut(M)$ converges to the identity $\id_M$ in the $u$-topology if and only if for all $\varphi \in M_*$ we have $\| \theta_i(\varphi) -\varphi \| \to 0$ as $i \to \infty$. This turns $\Aut(M)$ into a complete topological group. When $M_*$ is separable, $\Aut(M)$ is Polish. Since the standard form of $M$ is unique, the group $\Aut(M)$ also acts naturally on $\rL^2(M)$. Then the $u$-topology is also the topology of pointwise norm convergence on $\rL^2(M)$.

We denote by $\Ad : \mathcal{U}(M) \rightarrow \Aut(M)$ the continuous homomorphism which sends a unitary $u$ to the corresponding inner automorphism $\Ad(u)$. We denote by $\Inn(M)$ the image of $\Ad$. We denote by $\Out(M)=\Aut(M)/\Inn(M)$ the quotient group. Let $\pi_M : \Aut(M) \to \Out(M)$ be the quotient map. A net $(\theta_i)_{i \in I}$ in $\Aut(M)$ satisfies $\pi_M(\theta_i) \to 1$ in $\Out(M)$ as $i \to \infty$ if and only if there exists a net $(u_i)_{i \in I}$ in $\mathcal{U}(M)$ such that $\Ad(u_i) \circ \theta_i \to \id_M$ as $i \to \infty$. 

Note that a net of unitaries $(u_i)_{i \in I}$ in $\mathcal{U}(M)$ is centralizing in $M$ if and only if $\Ad(u_i) \to \id_M$ as $i \to \infty$. Hence, when $M$ is a full factor, the definitions imply that the map $\Ad : \mathcal{U}(M) \rightarrow \Aut(M)$ is open on its range. Therefore, $\Inn(M)$ is isomorphic as a topological group to the quotient $\mathcal{U}(M)/\{ z \in \C \mid |z|=1 \}$. In particular, $\Inn(M)$ is complete hence closed in $\Aut(M)$ and $\Out(M)$ is a complete topological group.

\subsection*{Ultraproducts von Neumann algebras}
Let $M$ be any $\sigma$-finite von Neumann algebra. Let $I$ be any nonempty directed set and $\omega$ any {\em cofinal} ultrafilter on $I$, i.e.\ $\{i \in I : i \geq i_0\} \in \omega$ for every $i_0 \in I$. When $I = \N$, $\omega$ is cofinal if and only if $\omega$ is {\em nonprincipal}, i.e.\ $\omega \in \beta \N \setminus \N$. Define
\begin{align*}
\mathcal I_\omega(M) &= \left\{ (x_i)_i \in \ell^\infty(I, M) \mid x_i \to 0\ \ast\text{-strongly as } i \to \omega \right\} \\
\mathfrak M^\omega(M) &= \left \{ (x_i)_i \in \ell^\infty(I, M) \mid  (x_i)_i \, \mathcal I_\omega(M) \subset \mathcal I_\omega(M) \text{ and } \mathcal I_\omega(M) \, (x_i)_i \subset \mathcal I_\omega(M)\right\}.
\end{align*}
The multiplier algebra $\mathfrak M^\omega(M)$ is a $\rC^*$-algebra and $\mathcal I_\omega(M) \subset \mathfrak M^\omega(M)$ is a norm closed two-sided ideal. Following \cite[\S 5.1]{Oc85}, the quotient $\rC^{*}$-algebra $M^\omega := \mathfrak M^\omega(M) / \mathcal I_\omega(M)$ is a von Neumann algebra, known as the {\em Ocneanu ultraproduct} of $M$. We denote the image of $(x_i)_i \in \mathfrak M^\omega(M)$ by $(x_i)^\omega \in M^\omega$. Throughout this paper, we will use the notation from \cite{AH12} for ultraproducts.

\section{Proof of Theorem \ref{main amenable}}
Before we prove Theorem \ref{main amenable}, we recall the following definition.
\begin{df}
Let $\sigma : \Gamma \curvearrowright M$ be an action of a discrete group $\Gamma$ on a von Neumann algebra $M$. We say that $\sigma$ is:
\begin{itemize}
\item [$(\rm i)$] \emph{ergodic} if the only elements of $M$ fixed by $\sigma$ are the scalars.
\item [$(\rm ii)$] \emph{strongly ergodic} if for every cofinal ultrafilter $\omega$ on every directed set $I$, the action $\sigma^\omega : \Gamma \curvearrowright M^\omega$ is ergodic.
\item [$(\rm iii)$] \emph{centrally ergodic} if for every cofinal ultrafilter $\omega$ on every directed set $I$, the action $\sigma_\omega : \Gamma \curvearrowright M_\omega$ is ergodic.
\end{itemize}
\end{df}

\begin{rem}
By a standard diagonal extraction argument, if an action $\sigma$ is strongly ergodic then $\sigma^\omega$ must also be strongly ergodic. Similarly, if $\sigma$ is centrally ergodic, then $\sigma_\omega$ is strongly ergodic. We also make the following observation: if $\mathcal{Z}(M)$ is discrete, then an action $\Gamma \curvearrowright M$ is strongly ergodic (resp.\ centrally ergodic) if and only if $\Gamma$ acts transitively on the minimal projections of $\mathcal{Z}(M)$ and for some (hence any) minimal projection $p \in \mathcal{Z}(M)$ the action of the stabilizer $\Gamma_p \curvearrowright pMp$ is strongly ergodic (resp.\ centrally ergodic). 
\end{rem}

It is well known that a pmp action of an amenable group on a diffuse probability space can never be strongly ergodic. The key result we need to prove Theorem \ref{main amenable} is the following noncommutative version. This result was very recently obtained in \cite[Theorem 4.1]{PSV18} in the case where $\sigma$ is free (and $\varphi$ is a trace). We will show that the general case can be reduced to the free case.

\begin{thm}\label{not strongly ergodic}
Let $\sigma : \Gamma \curvearrowright (M,\varphi)$ be a state preserving action of a discrete group $\Gamma$ on a von Neumann algebra $M$ with a faithful normal state $\varphi$. If $\Gamma$ is amenable and $\sigma$ is strongly ergodic, then $M$ is finite dimensional.
\end{thm}

In order to prove this non-strong ergodicity theorem, we first need to understand central ergodicity. We show that central ergodicity for group actions on $\II_1$ factors admits a useful spectral gap characterization which generalizes \cite[Theorem 2.1]{Co75}

\begin{prop} \label{gap centrally ergodic}
Let $\sigma : \Gamma \curvearrowright M$ be an action of a discrete group $\Gamma$ on a type $\II_1$ factor $M$. Then $\sigma$ is centrally ergodic if and only if the representation of $\mathcal{U}(M) \rtimes_\sigma \Gamma$ on $\rL^2(M)$ has spectral gap.
\end{prop}
\begin{proof}
Let $N=M \rtimes_\sigma \Gamma$. Let $\Sigma=\mathcal{U}(M) \cup \{u_g \mid g \in \Gamma \} \subset N$. A direct application of \cite[Theorem B]{Ma17} shows that we can find a projection $p \in M$ with $\tau(p) > \frac{1}{2}$, finite sets $S \subset \mathcal{U}(M)$ and $K \subset \Gamma$, and a constant $\kappa > 0$ such that for all $x \in pMp$ with $\tau(x)=0$ we have
\[ \|x\|_2 \leq \kappa \left( \sum_{u \in S} \|p(ux-xu)p \|_2 + \sum_{g \in K} \| p(u_gx-xu_g)p \|_2 \right) \] 
Let $v=2p-1 \in \mathcal{U}(M)$ and let $w \in \mathcal{U}(M)$ be any unitary which satisfies $w(1-p)w^* \leq p$. Let $S'=S \cup \{v,w \}$. Then it is not hard to check that there exists some constant $\kappa' > 0$ such that for all $x \in M$ we have
\[ \|x-\tau(x) \|_2 \leq \kappa' \left( \sum_{u \in S'} \|ux-xu \|_2 + \sum_{g \in K} \| \sigma_g(x)-x \|_2 \right) \]
\end{proof}

\begin{thm}\label{not centrally ergodic}
Let $\sigma : \Gamma \curvearrowright M$ be an action of a discrete group $\Gamma$ on a finite von Neumann algebra $M$. Suppose that $\Gamma$ is amenable. If $\sigma$ is centrally ergodic then $M$ is a direct sum of full factors.
\end{thm}
\begin{proof}
First observe that if $\sigma$ is centrally ergodic then its restriction to the center $\mathcal{Z}(M)$ must be strongly ergodic. Since $\Gamma$ is amenable, this only happens if $\mathcal{Z}(M)$ is discrete and $\Gamma$ acts transitively on the minimal projections of $\mathcal{Z}(M)$. Moreover, for any minimal projection $p \in \mathcal{Z}(M)$, the action of its stabilizer $\Gamma_p=\{g \in \Gamma \mid \sigma_g(p)=p\}$ on $pMp$ must be centrally ergodic. Hence the problem reduces to the case where $M$ is a $\II_1$ factor. So suppose that $M$ is now a $\II_1$ factor which is not full. We will show that the representation of $\mathcal{U}(M) \rtimes_\sigma \Gamma$ on $\rL^2(M)$ does not have spectral gap and this will contradict the conclusion of Proposition \ref{gap centrally ergodic}. Since $M$ is not full, we can find a net $(p_i)_{i \in I}$ of nonzero projections $p_i \in M$ such that $\lim_i p_i =0$ and $\lim_i \frac{1}{\tau(p_i) } \|up_i-p_i u \|_1=0$ for all $u \in \mathcal{U}(M)$. Fix $\varepsilon > 0$ and $K \subset \Gamma$ a finite subset. Then, since $\Gamma$ is amenable, we can find $F \subset \Gamma$ such that $| gF \triangle F | < \varepsilon |F|$ for all $g \in K$. Let 
\[ x_i=\left( \frac{1}{|F| \tau(p_i)}\sum_{ g \in F} \sigma_g(p_i) \right)^{1/2} \in M \]
For all $i$, we have $\|x_i \|_2=1$ and if we let $q_i=\mathrm{supp}(x_i)=\vee_{g \in F} \sigma_g(p_i)$, we have
 $$\tau(x_i) =\tau(x_iq_i) \leq \|x_i\|_2 \|q_i\|_2=\|q_i\|_2.$$
 Since $\lim_i q_i = 0$, we get $\lim_i \tau(x_i)=0$. We also have 
\[ \| ux_iu^{*}-x_i  \|_2^2 \leq \| ux_i^{2}u^{*}-x_i^{2}\|_1 \leq \frac{1}{|F|\tau(p_i)}\sum_{g \in F}\|\sigma^{-1}_g(u)p_i-p_i \sigma^{-1}_g(u) \|_1 \to 0 \] 
Moreover, for all $g \in K$, we have
 \[ \| \sigma_g(x_i)-x_i \|_2^2 \leq \| \sigma_g(x_i^{2})-x_i^{2} \|_1 \leq  \frac{|gF \triangle F|}{|F|}  \leq \varepsilon\]
 Since $K$ and $\varepsilon > 0$ are arbitrary, this shows that the representation of $\mathcal{U}(M) \rtimes_\sigma \Gamma$ on $\rL^2(M)$ does not have spectral gap, as we wanted.
\end{proof}

For the proof of Theorem \ref{not strongly ergodic}, we need a lemma to reduce to the finite case.

\begin{lem}
Let $M$ be any von Neumann algebra with a faithful normal state $\varphi$. If for some ultrafilter $\omega \in \beta \N \setminus \N$, we have that $M^\omega_{\varphi^{\omega}}$ is finite dimensional, then $M$ itself is finite dimensional.
\end{lem}
\begin{proof}
Suppose that $M^\omega_{\varphi^{\omega}}$ is finite dimensional. Then $M_\varphi$ is also finite dimensional. Let $e$ be a minimal projection in $M_\varphi$. Then $(eMe)_{e\varphi e}=\C$. Therefore, by \cite[Lemma 5.3]{AH12}, either $eMe=\C e$ or $eMe$ is a type $\III_1$ factor. But if $eMe$ is a type $\III_1$ factor, then $(eMe)^\omega_{(e\varphi e)^\omega}$ is a $\II_1$ factor by \cite[Theorem 4.20 and Proposition 4.24]{AH12}. This is not possible by assumption. Hence $eMe=\C e$. This shows that $e$ is also minimal in $M$ and we conclude easily that $M$ is finite dimensional. 
\end{proof}

\begin{proof}[Proof of Theorem \ref{not strongly ergodic}]
For the same reason as in the proof of Theorem \ref{not centrally ergodic}, we can reduce to the case where $M$ is a factor. Moreover, if $\sigma : \Gamma \curvearrowright M$ is strongly ergodic, then the action $\sigma^\omega : \Gamma \curvearrowright M^\omega_{\varphi^\omega}$ is also strongly ergodic and if $M^\omega_{\varphi^\omega}$ is finite dimensional then $M$ itself is finite dimensional. Therefore, the proof reduces easily to the case where $M$ is $\II_1$ factor and $\varphi$ is a trace. Note also that we can always find a finitely generated subgroup $G < \Gamma$ such that $\sigma|_G$ is still strongly ergodic. Hence we may assume that $\Gamma$ is countable. Let $H=\{ g \in \Gamma \mid \sigma_g \text{ is inner} \}$. Let $P$ be the von Neumann algebra generated by all the unitaries $u \in M$ for which there exists $g \in H$ such that $\sigma_g=\Ad(u)$. Note that $P$ is amenable because it is generated by an amenable group of unitaries. Moreover, it is globally invariant under the action $\sigma$, hence $\sigma|_P$ must also be strongly ergodic and in particular centrally ergodic. Therefore, by Theorem \ref{not centrally ergodic}, $P$ is a direct sum of full factors which means that $P$ is discrete since it is amenable. Let $Q=P' \cap M$. The action of $\Gamma$ on $Q$ is again strongly ergodic. Hence $Q$ has discrete center and for any minimal projection $p$ in the center of $Q$, the action of its stabilizer $\Gamma_p$ on $pQp$ is also strongly ergodic. Note that $H$ acts trivially on $Q$, hence we have a strongly ergodic action of $\Gamma_p/H$ on $pQp$, and we claim that this action of $\Gamma_p/H$ is outer. Indeed, if $g \in \Gamma_p$ is such that $\sigma_g|_{pQp}$ is inner, then $\sigma_g|pMp$ is also inner because $pPp$ is a finite dimensional factor and $pMp=pPp \otimes pQp$. Since $M$ is a factor, this implies that $\sigma_g$ is inner and therefore $g \in H$. This shows that the action of $\Gamma_p/H$ on $pQp$ is outer. Now, $pQp$ is a factor which admits a strongly ergodic outer action of an amenable group. Therefore $pQp$ is trivial by \cite[Theorem 4.1.$(\rm ii)$]{PSV18}. We conclude that $pMp=pPp \otimes pQp \cong pPp$ is a finite dimensional factor, hence $M$ is also a finite dimensional factor.
\end{proof}

\begin{proof}[Proof of Theorem \ref{main amenable}]
Let $I$ be the directed set of all neighborhoods of the identity in $\Aut(M)$ and take $\omega$ any cofinal ultrafilter on $I$. Let $P=M' \cap (M \rtimes_\sigma \Gamma)^\omega$. Note that $P$ is globally invariant under the action $\sigma^\omega : \Gamma \curvearrowright (M \rtimes_\sigma \Gamma)^\omega$. Since $M \rtimes_\sigma \Gamma$ is a full factor, then the action of $\Gamma$ on $P$ must be strongly ergodic. Moreover, the restriction to $P$ of the canonical faithful normal conditional expectation $\rE_M : (M \rtimes_\sigma \Gamma)^\omega \rightarrow M$ defines a canonical $\Gamma$-invariant faithful normal state on $P$. Since $\Gamma$ is amenable, Theorem \ref{not strongly ergodic} implies that $P$ is finite dimensional. Now, suppose that the kernel of $\overline{\sigma}$ is not finite or that its image is not discrete.  Then in both cases, we get that for every $\mathcal{V} \in I$, the set $\{g \in \Gamma \mid \overline{\sigma}(g) \in \pi_M(\mathcal{V})\}$ is infinite. Hence, we can find distinct elements $(g_{\mathcal{V},n})_{n \in \N}$ in $\Gamma$ and some unitaries $(u_{\mathcal{V},n})_{n \in \N}$ in $\mathcal{U}(M)$ such that $\sigma(g_{\mathcal{V},n}) \circ \Ad(u_{\mathcal{V},n}) \in \mathcal{V}$. Now let $w_n=(u_{g_{\mathcal{V},n}} u_{\mathcal{V},n})^\omega \in (M \rtimes_\sigma \Gamma)^\omega$ for all $n \in \N$. Then we have $w_n \in P$ for all $n \in \N$, and the family $(w_n)_{n \in \N}$ is orthonormal. This contradicts the fact that $P$ is finite dimensional. We conclude that $\overline{\sigma}$ has finite kernel and discrete image.
\end{proof}

\begin{rem}
Theorem \ref{not centrally ergodic} remains true without the finiteness assumption on $M$. Indeed, let $M$ be an arbitrary factor and $\sigma : \Gamma \curvearrowright M$ be a centrally ergodic action of an amenable group $\Gamma$. Then for any cofinal ultrafilter $\omega$ on any directed set $I$, we have that the action $\sigma_\omega : \Gamma \curvearrowright M_\omega$ is strongly ergodic. By Theorem \ref{not strongly ergodic}, this forces $M_\omega$ to be finite dimensional which actually implies that $M_\omega=\C$. Hence $M$ is full as we wanted. However, we do not know whether Theorem \ref{not strongly ergodic} remains true in the non-state preserving case.
\end{rem}

\section{Proof of Theorem \ref{main fullness}}

We will need the following spectral gap criterion.

\begin{thm}[{\cite[Theorem 4.4]{Ma16}}] \label{gap full}
Let $M$ be a full factor. Then there exists a normal state $\varphi$, a finite set $F \subset \Sigma_\varphi$ and a constant $\kappa > 0$ such that for all $x \in M$, we have
$$ \|x-\varphi(x)\|_\varphi \leq \kappa \sum_{\xi \in F} \| x \xi- \xi x \|.$$
\end{thm}

We will also need the following lemma.

\begin{lem} \label{neighborhood}
Let $M$ be a factor and $\varphi$ a normal state on $M$ (not necessarily faithful). Then for any neighborhood of the identity $\mathcal{V} \subset \Out(M)$, we can find a finite subset $F \subset \Sigma_\varphi$ and a constant $c > 0$ such that for all $\alpha \in \Aut(M) \setminus \pi_M^{-1}(\mathcal{V})$ and all $u \in \mathcal{U}(M)$, we have
$$ \sum_{ \xi \in F} \|u\alpha(\xi)-\xi u \|^2 \geq c.$$
\end{lem}
\begin{proof}
Suppose that there exists a net $(u_i)_{i \in I}$ in $\mathcal{U}(M)$ and a net $(\alpha_i)_{i \in I}$ in $\Aut(M) \setminus \pi_M^{-1}(\mathcal{V})$ such that $\|u_i\alpha_i(\xi)-\xi u_i \| \to 0$ for all $\xi \in \Sigma_\varphi$. Then if we put $\theta_i=\Ad(u_i)\circ \alpha_i$, we have $\theta_i(\xi) \to \xi$ for all $\xi \in \Sigma_\varphi$, hence for all $\xi \in p\rL^2(M)p$ where $p =\supp(\varphi)$. This contradicts \cite[Lemma 5.2]{Ma16}.
\end{proof}

\begin{thm}
Let $M$ be a full factor and $\sigma: \Gamma \curvearrowright M$ an action of a discrete group $\Gamma$ such that $M \rtimes_\sigma \Gamma$ is a factor. Suppose that the action $\Gamma \curvearrowright \partial_\sigma \Gamma$ has no invariant probability measure. Then there exists a normal state $\varphi$ on $M$, a finite set $F \subset \Sigma_\varphi$, a finite set $K \subset \Gamma$ and a constant $\kappa > 0$ such that for all $x \in M \rtimes_\sigma \Gamma$ we have
$$ \|x-\varphi(x) \|_\varphi \leq \kappa \left( \sum_{\xi \in F} \|x\xi-\xi x\| + \sum_{g \in K} \| xu_g-u_gx \|_\varphi^{\sharp} \right)$$
where we use the canonical conditional expectation $\rE : M \rtimes_\sigma \Gamma \rightarrow M$ to lift $\varphi$ to $M \rtimes_\sigma \Gamma$ and to view $\rL^2(M)$ as a subspace of $\rL^2(M \rtimes_\sigma \Gamma)$.
\end{thm}

\begin{proof}
Let $\varphi$ be any state on $M$ as in Theorem \ref{gap full}. Suppose that the inequality we want to prove does not hold. Then we can find a net $(x_i)_{i \in I}$ in $M$ with $\|x_i\|_\varphi \to 1$ and $\varphi(x_i) \to 0$ such that $\| x_i \xi - \xi x_i \| \to 0$ for all $\xi \in \Sigma_\varphi$ and $\|u_gx_iu_g^*-x_i \|_\varphi^\sharp \to 0$ for all $g \in \Gamma$. Let $x_i=\sum_{g \in \Gamma} x_i^g u_g$ be the Fourier decomposition of $x_i$. Up to some perturbation of $x_i$, we may assume that $x_i^g$ is invertible for all $i \in I$ and $g \in \Gamma$. Write $x_i^g=v_i^g |x_i^g|$ for the polar decomposition of $x_i^g$, with $v_i^g \in \mathcal{U}(M)$. Let $\lambda_i^g=\|x_i^g\|_{\sigma_g(\varphi)}$ and $\mu_i^g=\| x_i^*\|_\varphi$. Let $y_i=\sum_{g \in \Gamma} \lambda_i^gv_i^gu_g$. Let us show that $\|x_i-y_i\|^\sharp_\varphi \to 0$. For all $\xi \in \Sigma_\varphi$, we have
$$\sum_{g \in \Gamma} \| |x_i^g| \sigma_g(\xi) - \sigma_g(\xi) |x_i^g| \|^2+ \| |(x_i^g)^*| \xi - \xi |(x_i^g)^*|  \|^2 \leq 2 \sum_{g \in \Gamma} \|x_i^g \sigma_g(\xi) - \xi x_i^g \|^2=2\|x_i \xi - \xi x_i \|^2 \to 0.$$
This implies that
$$  \sum_{g \in \Gamma} \| |x_i^g|-\lambda_i^g \|_{\sigma_g(\varphi)}^2+\| |(x_i^g)^*|-\mu_i^g \|_\varphi^2 \to 0.$$
$$ \sum_{g \in \Gamma} \|x_i^gu_g-\lambda_i^gv_i^gu_g\|_\varphi^2+\|u_g^*(x_i^g)^*-\mu_i^gu_g^*(v_i^g)^*\|_\varphi^2 \to 0.$$ 
Moreover, we have
$$ \sum_g | \lambda_i^g-\mu_i^g|^2 \leq \sum_{g \in \Gamma} \| x_i^g \sigma_g(\xi_\varphi)-\xi_\varphi x_i^g \|^2 =\|x_i \xi_\varphi-\xi_\varphi x_i \|^2 \to 0.$$
Therefore, we obtain 
$$ \left( \|x_i-y_i\|^\sharp_\varphi \right)^2 =  \sum_{g \in \Gamma} \|x_i^gu_g-\lambda_i^gv_i^gu_g\|_\varphi^2+\|u_g^*(x_i^g)^*-\lambda_i^gu_g^*(v_i^g)^*\|_\varphi^2 \to 0.$$ 
Since $\|x_i-y_i\|^\sharp_\varphi \to 0$, we have $\lim_i \|y_i \xi -\xi y_i \| = \lim_i \| x_i \xi -\xi x_i \|=0$ for all $\xi \in \Sigma_\varphi$. And since  for all $g \in \Gamma$, the net $(u_gx_iu_g^*)_{i \in I}$ satisfies the same assumption as $(x_i)_{i \in I}$, we also obtain $\|u_gx_iu_g^*-u_gy_iu_g^*\|_\varphi^\sharp \to 0$ for all $g \in \Gamma$. In particular, we have $\|u_gy_iu_g^*-y_i\|_\varphi^\sharp \to 0$ for all $g \in \Gamma$. 

Now, let
$$ \eta_i= \sum_{g \in \Gamma} \lambda_i^gu_g \in \ell^2(\Gamma).$$
Since $\|u_gy_iu_g^*-y_i\|_\varphi \to 0$, we have $\|u_g\eta_i u_g^*-\eta_i \| \to 0$ for all $g \in \Gamma$. Define a state $\Psi$ on $\ell^\infty(\Gamma)=C(\beta \Gamma)$ by taking any weak$^*$ accumulation point of $\Psi_i=\langle \cdot \eta_i, \eta_i \rangle, \: i \in I$. Then the state $\Psi$ is conjugacy invariant. We will reach a contradiction by showing that $\Psi$ is supported on $\partial_\sigma \Gamma$.

First we show that $\Psi$ is supported on $\beta \Gamma \setminus \Gamma$. For this, we have to show that $\lim_i \lambda_i^g=0$ for all $g \in \Gamma$. Let $\omega$ be a cofinal ultrafilter on $I$. Let $\lambda^g=\lim_{i \to \omega} \lambda_i^g$. Observe that $g \mapsto \lambda^g$ is constant on each conjugacy class of $\Gamma$. Since $(\lambda_i^g)^2\| \Ad(v_i^g) \sigma_g(\xi)-\xi\|^2 \to 0$ for all $\xi \in \Sigma_\varphi$, then for every $g$ such that $\lambda^g \neq 0$, we must have $\lim_{i \to \omega} \Ad(v_i^g)\sigma_g(\xi)=\xi$ for all $\xi \in \Sigma_\varphi$. Since $M$ is full this implies that $\sigma_g$ is inner and that there exists a unique unitary $v^g \in \mathcal{U}(M)$ such that $\sigma_g=\Ad(v^g)$ and $\lim_{i \to \omega} v_i^gv^gp=p$ in the $*$-strong topology, where $p=\supp(\varphi)$. By the uniqueness of $v^g$ we must have $\sigma_h(v^g)=v^{h^{-1}gh}$ for all $h \in \Gamma$. Therefore, we have
$$y=\sum_{ \substack{g \in \Gamma \\ \lambda^g \neq 0}} \lambda^g(v^g)^*u_g \in \mathcal{Z}(M \rtimes_\sigma \Gamma).$$
Since, by assumption, $M \rtimes_\sigma \Gamma$ is a factor, we conclude that $y=\varphi(y)=0$. This shows that $\lambda^g=0$ for all $g \in \Gamma$ and this holds for any cofinal ultrafilter $\omega$ on $I$. Hence, we obtain $\lim_i \lambda_i^g=0$ for all $g \in \Gamma$.

Now, we show that $\Psi$ is supported on $\{ x \in \beta \Gamma \mid \lim_{g \to x } \overline{\sigma}(g)=1 \}$. Take $\mathcal{V}$ any neighborhood of the identity in $\Out(M)$. By Lemma \ref{neighborhood}, there exists a constant $c > 0$, a finite set $F \subset \Sigma_\varphi$ such that for all $\alpha \in \Aut(M) \setminus \pi_M^{-1}(\mathcal{V})$ and all $u \in \mathcal{U}(M)$, we have
$$ \sum_{ \xi \in F} \|u\alpha(\xi)-\xi u \|^2 \geq c.$$
This implies that
$$ \sum_{ \substack{g \in \Gamma \\ \overline{\sigma}(g) \notin \mathcal{V}}} (\lambda_i^g)^2 \leq \frac{1}{c} \sum_{g \in \Gamma} \sum_{\xi \in F} (\lambda_i^g)^2 \| v_i^g \sigma_g(\xi)-\xi v_i^g\|^2 =\sum_{\xi \in F} \| y_i \xi- \xi y_i\|^2 \to 0.$$
We conclude that $\Psi$ is supported on $\{ x \in \beta \Gamma \mid \lim_{g \to x } \overline{\sigma}(g)=1 \}$ and therefore on $\partial_\sigma \Gamma = \{ x \in \beta \Gamma \setminus \Gamma \mid \lim_{g \to x } \overline{\sigma}(g)=1 \}$.

\end{proof}

\bibliographystyle{plain}

\begin{thebibliography}{MvN43}

\bibitem[AH12]{AH12} {\sc H. Ando, U. Haagerup}, {\it Ultraproducts of von Neumann algebras.} J. Funct. Anal. {\bf 266} (2014), 6842--6913.
%



\bibitem[Ch81]{Ch81} {\sc M. Choda}, {\it Inner amenability and fullness.} Proc. Amer. Math. Soc. {\bf 86} (1982), 663--666.

%
\bibitem[Co74]{Co74} {\sc A. Connes}, {\it Almost periodic states and factors of type ${\rm III_1}$.} J. Funct. Anal. {\bf 16} (1974), 415--445.


\bibitem[Co75]{Co75} {\sc A. Connes}, {\it Classification of injective factors. Cases ${\rm II_1}$, ${\rm II_\infty}$, ${\rm III_\lambda}$, $\lambda \neq 1$.} Ann. of Math. {\bf 74} (1976), 73--115.



\bibitem[Ha73]{Ha73} {\sc U. Haagerup}, {\it The standard form of von Neumann algebras.} Math. Scand. {\bf 37} (1975), 271--283.





\bibitem[H14]{H14} {\sc C. Houdayer}, {\it  Structure of $\II_1$ factors arising from free Bogoljubov actions of arbitrary groups.} Adv. Math. {\bf 260} (2014), 414--457.


\bibitem[HI15]{HI15} {\sc C. Houdayer, Y. Isono}, {\it Bi-exact groups, strongly ergodic actions and group measure space type ${\rm III}$ factors with no central sequence.} Comm. Math. Phys. {\bf 348} (2016), 991--1015.

\bibitem[HI18]{HI18} {\sc C. Houdayer, Y. Isono}, {\it Factoriality, Connes? type $\III$ invariants and fullness of amalgamated free product von Neumann algebras.} To appear in Proc. Roy. Soc. Edinburgh Sect. A.


%
%


\bibitem[Jo81]{Jo81} {\sc V.F.R. Jones}, {\it Central sequences in crossed products of full factors.} Duke Math. J. {\bf 49} (1982), 29--33.




\bibitem[Ma16]{Ma16} {\sc A. Marrakchi}, {\it Spectral gap characterization of full type $\mathrm{III}$ factors.} {To appear in J. Reine Angew. Math.} {\tt arXiv:1605.09613}

\bibitem[Ma17]{Ma17} {\sc{A.~Marrakchi}}, {\it Strongly ergodic actions have local spectral gap.} {To appear in Proc. Amer. Math. Soc.} {2017, \tt{	arXiv:1707.00438.}}

\bibitem[MvN43]{MvN43} {\sc F. Murray, J. von Neumann}, {\it Rings of operators.} ${\rm IV}$.  Ann. of Math. {\bf 44} (1943), 716--808.

\bibitem[Oc85]{Oc85} {\sc A. Ocneanu}, {\it Actions of discrete amenable groups on von Neumann algebras.} Lecture Notes in Mathematics, {\bf 1138}. Springer-Verlag, Berlin, 1985. iv+115 pp.

\bibitem[Oz16]{Oz16} {\sc N. Ozawa}, {\it A remark on fullness of some group measure space von Neumann algebras.} Compos. Math. {\bf 152} (2016), 2493--2502. 




\bibitem[PSV18]{PSV18} {\sc S. Popa, D. Shlyakhtenko, S. Vaes}, {\it Classification of regular subalgebras of the hyperfinite $\II_1$ factor.}  Preprint.







%
%




\bibitem[VV14]{VV14} {\sc S. Vaes, P. Verraedt}, {\it Classification of type ${\rm III}$ Bernoulli crossed products.} Adv. Math. {\bf 281} (2015), 296--332.

\end{thebibliography}

\end{document}